\documentclass[letterpaper, 10 pt, conference]{ieeeconf}  % Comment this line out
\IEEEoverridecommandlockouts                              % This command is only
\usepackage[dvipdfmx]{graphics} % for pdf, bitmapped graphics files
\usepackage{epsfig} % for postscript graphics files
\usepackage{amsmath} % assumes amsmath package installed
\usepackage{amssymb}  % assumes amsmath package installed
\usepackage{color} 
\usepackage{cite} 

\newtheorem{thm}{Theorem}
\newtheorem{cor}[thm]{Corollary}
\newtheorem{lem}[thm]{Lemma}

\newtheorem{prop}[thm]{Proposition}

\newtheorem{rem}[thm]{Remark}

\allowdisplaybreaks[1]

\title{\LARGE \bf
A proposal of adaptive parameter tuning for robust stabilizing control of $N$--level quantum angular momentum systems
}

\author{Shoju Enami and Kentaro Ohki% <-this % stops a space
\thanks{This work was supported by JSPS KAKENHI Grant Number JP19K03619 and JP20H02168.}% <-this % stops a space
\thanks{S. Enami and K. Ohki are with Department of Applied Mathematics and Physics, Graduate School of Informatics, 
        Kyoto University, 606-8501 Yoshida-Honmachi, Sakyo-ku, Kyoto, Japan.
        {\tt\small ohki@bode.amp.i.kyoto-u.ac.jp}}%
}

\begin{document}

\maketitle
\thispagestyle{empty}
\pagestyle{empty}

%%%%%%%%%%%%%%%%%%%%%%%%%%%%%%%%%%%%%%%%%%%%%%%%%%%%%%%%%%%%%%%%%%%%%%%%%%%%%%%%
\begin{abstract}
Stabilizing control synthesis is one of the central subjects in control theory and engineering, and it always has to deal with unavoidable uncertainties in practice. In this study, we propose an adaptive parameter tuning algorithm for robust stabilizing quantum feedback control of $N$-level quantum angular momentum systems with a robust stabilizing controller proposed by 
[Liang, Amini, and Mason, {\em SIAM J. Control Optim.}, 59 (2021), pp. 669-692].  
The proposed method ensures local convergence to the target state. 
Besides, numerical experiments indicate its global convergence if the learning parameters are adequately determined.  

\end{abstract}

%%%%%%%%%%%%%%%%%%%%%%%%%%%%%%%%%%%%%%%%%%%%%%%%%%%%%%%%%%%%%%%%%%%%%%%%%%%%%%%%
\section{Introduction}

Stabilizing controller synthesis is one of the central problems in control systems, even if systems are described by quantum mechanics \cite{Ramon_IEEE2005,Mirrahimi_stabilizing_FB_control,tsumura2007global,sarlette2012stabilization,ticozzi2012hamiltonian,ticozzi2013stabilization,scaramuzza2015switching,ticozzi2017alternating,liang2019exponential,cardona2020exponential,wen2021exponential}.  
Unfortunately, stabilizing control is vulnerable to failure due to the existence of uncertainties in practice.   
There are two conventional approaches 
to overcome this problem; 
robust control \cite{zhou1996rao,petersen2000robust} and adaptive control \cite{narendra1989stable,krstic1995nonlinear}.  
Robust control ensures performance of the control system under worst-case scenario against a given set of uncertainties.  
It has been actively studied for quantum systems\cite{james2004rso,james2005quantum,dong2015robust,vladimirov2018risk,James_H_infty,LiangAminiMason2021}, 
as well as classical systems.   
The most common problem of robust control is that it is difficult to specify the uncertainties in advance, and even if possible, the robust controller tends to yield conservative control performance.  
On the other hand, adaptive control operates on the system by learning model parameters.  
Adaptive approaches for quantum system identification and filtering have also been studied \cite{bonnabel2009observer,LeghtasMirrahimiRouchon2011,gupta2021adaptive}.  
However, these studies do not consider a stochastic continuous measurement signal, which is known as a homodyne measurement signal in physics and one of the commonly used detection models for quantum physics, and no real-time adaptive control framework has been proposed in the previous studies so far.

Recently, Liang {\it et. al.} \cite{LiangAminiMason2021} derived certain conditions for robust stabilization of $N$--level quantum angular momentum systems with uncertain parameters and initial state.  
They revealed that the accurate estimation of the multiplication of the two parameters is essential for their robust stabilization.  
This fact is important because it ensures robust stabilization by accurate estimation of the multiplication of the parameters only rather than the parameters individually. 
Motivated by \cite{LiangAminiMason2021}, we propose an adaptive parameter tuning algorithm with stabilizing control.  
To the best of our knowledge, this is the first study on adaptive control for quantum systems with continuous-time measurement feedback.  
The proposed adaptive law is aimed to minimize the difference between the original and model outputs. 
The method is simple, but it works well in numerical experiments under certain assumptions, and ensure local convergence.

\subsection{Contributions}

The contributions of this study are summarized as follows:
\begin{itemize}
\item An adaptive parameter tuning algorithm for robust quantum stabilizing control is proposed (Equation \eqref{eq:adaptive_tuning}). 

\item An asymptotic property of the estimate under steady-state conditions is derived (Proposition \ref{prop:statistical_property_tuning}).  

\item Local convergence of the proposed method is evaluated under certain assumptions (Theorem \ref{thm:main_result}).  

\end{itemize}

\subsection{Organization}

The rest of this paper is organized as follows.  
The problem is stated in Section \ref{sec:problem_setting}.  
In Sec. \ref{sec:main_results}, we propose an adaptive parameter tuning algorithm and the analytical results are shown.   
The proposed method is evaluated numerically and compared with the application of \cite{LiangAminiMason2021} in Sec. \ref{sec:numerical_experiments}.  
We conclude the paper in Sec. \ref{sec:conclusion}.

\subsection{Notation} 

	$\mathbb{N}$, $\mathbb{R}$ and $\mathbb{C}$ are natural, real, and complex numbers, respectively,   
	and $\mathrm{i}:= \sqrt{-1}$.  
	$\mathbb{R}^{n\times m}$ and $\mathbb{C}^{n\times m}$ are real and complex $n \times m$ matrices, respectively.  
	$X^{\ast}$ implies the Hermitian conjugate of matrix $X$.  
	We use $I_{n} \in \mathbb{C}^{n\times n}$ as the identify matrix.  
	For $X = X^{\ast} \in \mathbb{C}^{n \times n}$, $X>0$ ($X\geq 0$) indicates that $X$ is a positive-(semi)definite matrix.  
	When two positive-semidefinite matrices $X$ and $Y$ satisfy $X = Y^{2}$, we denote $Y = \sqrt{X}$.  
	The absolute value of a square matrix is defined as $|X| := \sqrt{X^{\ast} X}$ and for $X \in \mathbb{C}^{n\times n}$, the trace norm is defined as $\| X \| _{\rm Tr} := \mathrm{Tr}[ |X|]$.  
	$\mathcal{S}(\mathbb{C}^{n}) := \{ \rho \in \mathbb{C}^{n\times n} \ | \ \rho = \rho^{\ast} \geq 0 , \ \mathrm{Tr}[\rho ]=1\}
$. 
	Denote $[X , Y ]_{-} := XY - YX$ $\forall X,Y \in \mathbb{C}^{n\times n} $.  
	$\mathbb{E}_{w}$ indicates the expectation in terms of a random variable or a stochastic process $w$.  
	$O(\varepsilon)$ is Landau's $O$ as $\varepsilon \to 0$.

%%%%%%%%%%%%%%%%%%%%%%%%%%%%%%%%%%%%%%%%%%%%%%%%%%%%%%%%%%%%%%%%%%%%%%%%%%%%%%%%
\section{Problem Formulation}
\label{sec:problem_setting}

\subsection{Measurement-based Feedback Quantum Systems}

Let $J \in \mathbb{N}$ and $N:=2J+1$, and let us consider the following quantum stochastic differential equation \cite{Mirrahimi_stabilizing_FB_control,liang2019exponential,LiangAminiMason2021}.  
\begin{align}
d\rho(t) =& \mathrm{i} [H_{\omega}(u(t)) , \rho (t) ]_{-} dt - \frac{M}{2} [J_{z} , [J_{z} , \rho (t) ]_{-} ]_{-} dt
\nonumber \\
& + \sqrt{\eta M} \left( J_{z} \rho (t) + \rho(t) J_{z} - 2\mathrm{Tr}[ J_{z} \rho (t)  ] \rho (t) \right)
\nonumber \\ & \quad \quad \times \left( dy(t) - 2\sqrt{\eta M} \mathrm{Tr}[ J_{z} \rho (t)  ]dt \right)
\label{eq:sme}
\end{align}
with an initial state $\rho (0) \in \mathcal{S}(\mathbb{C}^{N})$, 
where $\rho (t) \in \mathcal{S}(\mathbb{C}^{N})$ is a conditional state of the system, 
$u(t)$ is the  control input, $y(t)$ is the measurement output, 
$H_{\omega}(u):= \omega J_{z} + u J_{y}$, 
\begin{align*}
\\
J_{z} :=&
\mathrm{diag}(J, J-1,\ \dots , \ -J+1 , -J),
\\
J_{y}:=&
\begin{bmatrix}
0 & -\mathrm{i}  c_{1} & 0 & \cdots & 0
\\
\mathrm{i}  c_{1} & \ddots & \ddots  & & \vdots
\\
0& \ddots & \ddots  & \ddots & 0
\\
\vdots & & \ddots & \ddots & -\mathrm{i}c_{N-1}
\\
0& \cdots & 0& \mathrm{i}c_{N-1} & 0
\end{bmatrix}
,
\end{align*}
$c_{m} =\frac{1}{2}\sqrt{(2J+1 -m)m}$, $m=1,\dots ,N-1$, 
and $\omega >0$, $M >0$ is the coupling constant, and $\eta \in (0,1]$ denotes measurement efficiency \cite{Bouten_quant_filtering}.  
$u(t)$ is control input and throughout this paper, $u$ is assumed bounded.  
$\rho (t)$ is called the {\it state}, which is a quantum counterpart of (conditional) probability law.  
Equation \eqref{eq:sme} is called {\it stochastic master equation} having $N$ different equilibrium points if the control input $u(t)=0$.  
We denote each equilibrium point as $\rho _{n} \in \mathcal{S}(\mathbb{C}^{N})$, $n=0,\dots ,2J$, and the target state is described by $\rho _{\bar{n}}$.  
Note that $\rho _{n}$ consists of an eigenvector of $J_{z}$, i.e., 
\begin{align*}
J_{z} \rho _{n} = \rho _{n}J_{z} = (J-n) \rho _{n} \quad \forall n\in \{ 0,\dots ,2J\}
.
\end{align*}

A stabilization problem of \eqref{eq:sme} is to ensure that the state converges to a desirable equilibrium point.   
Therefore, model uncertainty must be considered in practical situations as 
different uncertainties exist  
in the model, initial state, and parameters of \eqref{eq:sme}.  
In this paper, we consider parametric uncertainty and unknown initial condition.  
Then, the nominal model is described as follows.  
\begin{align}
d\hat{\rho}(t) =& \mathrm{i} [H_{\hat{\omega}}(u(t)) , \hat{\rho} (t) ]_{-} dt - \frac{\hat{M}}{2} [J_{z} , [J_{z} , \hat{\rho} (t) ]_{-} ]_{-} dt
\nonumber \\
& + \sqrt{\hat{\eta} \hat{M}} \left( J_{z} \hat{\rho} (t) + \hat{\rho}(t) J_{z} - 2\mathrm{Tr}[ J_{z} \hat{\rho} (t)  ] \hat{\rho} (t) \right)
\nonumber \\ & \quad \quad \times \left( dy(t) - 2\sqrt{\hat{\eta} \hat{M}} \mathrm{Tr}[ J_{z} \hat{\rho} (t)  ]dt \right)
\label{eq:nominal_sde}
\end{align}
with its initial state $\hat{\rho} (0) \in \mathcal{S}(\mathbb{C}^{N})$.  
The differences from \eqref{eq:sme} are the initial state $\hat{\rho}(0)$ and the parameters $(\hat{\omega}, \hat{M}, \hat{\eta})$.  
Because the accessible state is $\hat{\rho}(t)$,  
the goal of the stabilization problem is to find a feedback controller $u(t) = u_{FB}(\hat{\rho}(t) )$ that ensures 
$\lim _{t \to \infty}\rho(t) = \rho _{\bar{n}}$ as one of the stochastic convergences.

\subsection{Previous Work}

	Liang {\it et. al.} \cite{LiangAminiMason2021} found certain sufficient conditions when the nominal state stabilization becomes the true state stabilization.  
	 One of their main results is that if the ratio of $\hat{\eta}\hat{M}$ and $\eta M$ is close to $1$, then there exists a stabilizing controller.    
	For convenience, we write $\hat{\theta}:= \sqrt{\hat{\eta} \hat{M}}$ and $\theta := \sqrt{\eta M}$, and then the following result holds \cite{LiangAminiMason2021}.  
	\begin{thm}[{\cite[Propositions 4.16 and 4.18]{LiangAminiMason2021}}]
	\label{thm:LiangAminiMason2021}

	Suppose $\hat{\theta}$ satisfies 
	\begin{align}
	\alpha _{\bar{n}} <  \frac{\hat{\theta}}{\theta} - 1 < \beta _{\bar{n}}, 
	\label{eq:convergence_condition}
	\end{align}
	where
	\begin{align*}
	\alpha_{\bar{n}} := &
	\left\{
	\begin{array}{ll}
	\displaystyle 
	-\frac{1}{2N-1}, & \bar{n} \in \{ 0, 2J \} , 
	\\
	\displaystyle
	-\frac{1}{N-2}, & \bar{n} =J, 
	\\ 
	\displaystyle
	-\frac{1}{L_{\bar{n}} +1}, & \mbox{otherwise},
	\end{array}
	\right.
	\\
	\beta _{\bar{n}} :=& 
	\left\{
	\begin{array}{ll}
	\displaystyle
	\frac{1}{2} \left(
	\sqrt{\frac{N+1}{N-1}} -1
	\right) , & \bar{n} \in \{ 0, 2J \} ,
	\\
	\displaystyle
	\frac{1}{N-2}, & \bar{n} =J,
	\\
	\displaystyle
	\frac{1}{L_{\bar{n}} -1}, & \mbox{otherwise}, 
	\end{array}
	\right.
	\end{align*}
	and $L_{\bar{n}} := 4|J-\bar{n}| \max \{ \bar{n}, 2J-\bar{n} \}$, $\hat{\rho}(0)$ is positive-definite, and $\rho (0) \in \mathcal{S}(\mathbb{C}^{N})$.  
	Then, there exists an asymptotically stabilizing control law that ensures 
	\begin{align*}
	(\rho (t) , \hat{\rho}(t)) \xrightarrow{t \to \infty} (\rho _{\bar{n}}, \rho _{\bar{n}} )
	\quad \mbox{a.s.}  
	\end{align*}

\end{thm}
	
	Note that Theorem \ref{thm:LiangAminiMason2021} is only part of their results.  
	See \cite{LiangAminiMason2021} for the details.

\subsection{Problem Statement}

Before stating our problem, we present a minor modification of Theorem \ref{thm:LiangAminiMason2021}.  

\begin{cor}
\label{cor:LiangAminiMason2021}
	
	Let $\hat{\theta}(t):= \sqrt{\hat{\eta}(t) \hat{M}(t)}$ be a time varying parameter and we assume that there exists $t_{0} >0$ that satisfies the following constraint; 
	\begin{align}
	\alpha _{\bar{n}} <  
	\frac{\hat{\theta}(t)}{\theta} -1 
	< \beta _{\bar{n}} \quad \forall t \geq t_{0},
	\label{eq:convergence_condition_tv}
	\end{align}
	where $\alpha _{\bar{n}}$ and $\beta _{\bar{n}}$ are the same as defined in Theorem \ref{thm:LiangAminiMason2021}, 
	$\hat{\rho}(t_{0}) >0$, 
	and $\rho (0) \in \mathcal{S}(\mathbb{C}^{N})$.  
	Then, there exists a stabilizing control law.

\end{cor}

\begin{proof}
	The proof is the same as \cite[Propositions 4.16 and 4.18]{LiangAminiMason2021}, so we omit it here.  
\end{proof}
 
 Corollary \ref{cor:LiangAminiMason2021} implies that if we can set the parameter $\hat{\theta}(t)$ appropriately, 
 the stabilization is achieved even if the initial parameter $\hat{\theta}(0)$ does not satisfy the condition \eqref{eq:convergence_condition}.  
 Therefore, the problem we deal with is how to estimate $\theta$ while stabilizing the state $\rho(t)$.  
Adaptive parameter tuning $\hat{\theta}(t)$ with stabilizing control is a simple and useful solution for the problem, as shown in the next section.

%%%%%%%%%%%%%%%%%%%%%%%%%%%%%%%%%%%%%%%%%%%%%%%%%%%%%%%%%%%%%%%%%%%%%%%%%%%%%%%%
\section{Proposed Method and Theoretical Results}
\label{sec:main_results}

Owing to the work of Liang \textit{et al.}\cite{LiangAminiMason2021}, there is an acceptable uncertainty of the parameter $\theta$ that ensures the convergence of the state to the target state.  
Hence, we only focus on the parameter tuning of $\hat{M}(t)$.   
The adaptive model is then described as follows (Fig. \ref{fig_adaptive controller}).  
\begin{align}
d\hat{\rho}(t) =& \mathrm{i} [H_{\hat{\omega}}(u(t)) , \hat{\rho} (t) ]_{-} dt - \frac{\hat{M}(t)}{2} [J_{z} , [J_{z} , \hat{\rho} (t) ]_{-} ]_{-} dt
\nonumber \\
& + \sqrt{\hat{\eta} \hat{M}(t)} \left( J_{z} \hat{\rho} (t) + \hat{\rho}(t) J_{z} - 2\mathrm{Tr}[ J_{z} \hat{\rho} (t)  ] \hat{\rho} (t) \right)
\nonumber \\ & \quad \quad \times \left( dy(t) - 2\sqrt{\hat{\eta} \hat{M}(t)} \mathrm{Tr}[ J_{z} \hat{\rho} (t)  ]dt \right)
,
\label{eq:adaptive_sde}
\end{align}
where $(\hat{\omega}, \hat{M}(0), \hat{\eta})$ are given and $\hat{M}(t)$ is calculated by our proposed parameter tuning algorithm below.  

	\begin{figure}[thpb]
       \centering
       \includegraphics[scale=0.45]{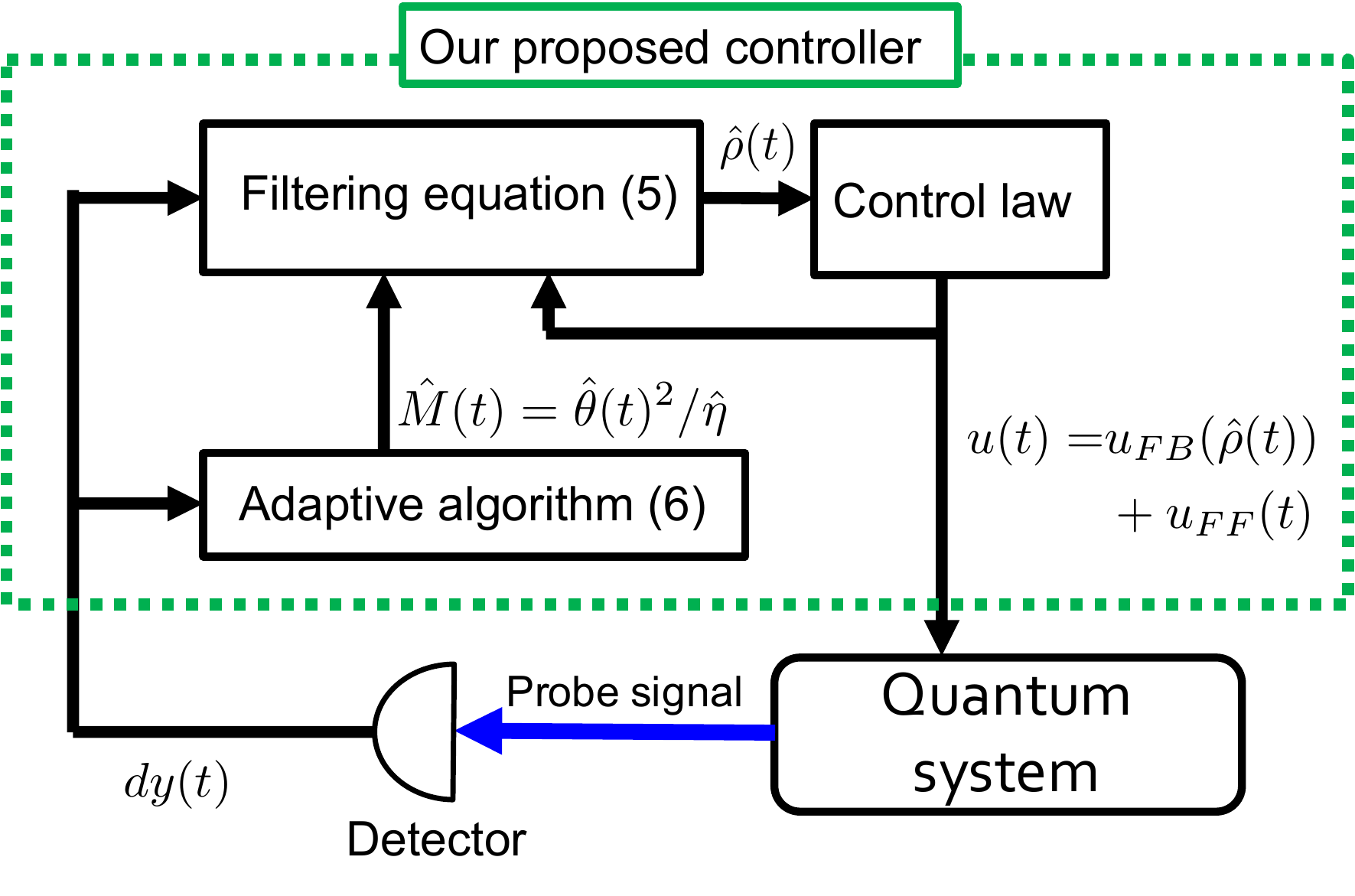}
       \caption{
       The proposed adaptive controller
       }
       \label{fig_adaptive controller}
       \end{figure}

\subsection{Proposed Adaptive Parameter Tuning Method}

For convenience, we use $\hat{\theta}(t) := \sqrt{\hat{\eta} \hat{M}(t)}$, $\hat{x}(t):= \mathrm{Tr}[J_{z} \hat{\rho}(t)]$, and $x(t) :=  \mathrm{Tr}[J_{z} \rho(t)]$.  
Then, we propose the following parameter tuning algorithm.  
\begin{align}
d\hat{\theta} (t) =& f(t) \left\{ - \hat{x}(t) ^{2} \hat{\theta}(t) dt + \frac{1}{2} \hat{x}(t) dy(t) \right\} ,
\label{eq:adaptive_tuning}
\\
f(t):= &(Kt +1)^{-p} ,\quad t\geq 0, 
\label{eq:learning_coef}
\end{align}
where $p\in (0,1]$ and $K>0$.  
Note that we update $\hat{M}(t)$ as $\hat{M}(t) = \hat{\theta}(t) ^2 / \hat{\eta}$.  
From the filtering theory \cite{Bouten_quant_filtering,bain2008fundamentals}, $dy(t) $ can be replaced by $dw(t) + 2 \theta x(t)$, where $w(t)$ is a standard Wiener process, and \eqref{eq:adaptive_tuning} then gives 
\begin{align*}
d\hat{\theta} (t) =& f(t) \hat{x}(t) \left\{ (\theta x(t) -  \hat{\theta}(t) \hat{x}(t) ) dt + \frac{1}{2} dw(t) \right\}
.
\end{align*}
If the noise $w(t)$ is removed, 
updating $\hat{\theta}(t)$ by Eq. \eqref{eq:adaptive_tuning} implies the same as instant gradient method of the cost function $| \theta x(t) - \hat{\theta}(t) \hat{x}(t)|^{2}$ with the weight $f(t)$.  
This is a type of Robbins-Monro algorithm for continuous-time problems \cite{RobbinsMonro1951,chen1994continuous,kushner2010stochastic}.  
Clearly, if $x(t) = \hat{x}(t) \neq 0$ for all $t\geq 0$, then the parameter tuning law is the continuous-time Robbins-Monro algorithm, which is guaranteed to converge to the true parameter.  
Unfortunately, 
 the assumption $x(t) = \hat{x}(t) \neq 0$ $\forall t\geq 0$ may not hold; therefore, 
we need to seek the condition when the parameter $\hat{\theta}(t)$ converges to the region described by \eqref{eq:convergence_condition_tv}.  
Note that the true parameter $\theta$ cannot be an equilibrium point of the system \eqref{eq:adaptive_tuning}. 
Thus, the noise $w$ is unavoidable, so we examine how to choose the parameter $(K, p)$ and obtain the accurate estimate asymptotically.

\begin{rem}

Because each unknown parameter is a positive constant, the adaptive parameter $\hat{\theta}(t)$ must to be positive.  
However, the solution of \eqref{eq:adaptive_tuning} is not ensured to be positive, 
so when $\hat{\theta}(t) $ becomes negative, we replace it with $0$ or a small positive number in practical implementations.  

\end{rem}

\subsection{Asymptotic Property of the Estimate}
	
	Here, we describe that the choice of the parameters $p$ and $K$ of \eqref{eq:learning_coef} is valid from the following proposition.  
	For convenience, we write $\mathbb{E}_{w} [ \bullet ] \equiv \mathbb{E}_{w} [ \bullet | \rho (0), \hat{\rho}(0) ]$.  
	
	\begin{prop}
	\label{prop:statistical_property_tuning}
	
	Suppose that a pair of initial states $(\rho (0), \hat{\rho}(0))$ is 
	in some $(\rho _{n},\rho_{m})$, $m\neq J$, and $u(t)=0$, and considering \eqref{eq:adaptive_tuning} with $p \in \mathbb{R}$ and $K>0$, the followings hold.  
	\begin{enumerate}
	\item For the mean of the $\hat{\theta}(t)$, 
	\begin{align*}
	\lim _{t \to \infty} 
	\mathbb{E}_{w}[\hat{\theta}(t)] =&
	\left\{
	\begin{array}{ll}
	\mbox{(depend on $\hat{\theta}(0)$)} ,& p>1,
	\\
	\theta \frac{J-n}{J-m},   & p \in (- \infty , 1].
	\end{array}
	\right.
	\end{align*}
	
	\item For the variance of the $\hat{\theta}(t)$, $V(\hat{\theta}(t)) := \mathbb{E}_{w}[ (\hat{\theta}(t) - \mathbb{E}_{w}[\hat{\theta}(t)] )^{2} ]$, 
	\begin{align*}
	\limsup _{t \to \infty}
	V(\hat{\theta}(t)) 
	\leq & \ 
	\frac{1}{8}, \quad  p>1,
	\\
	\lim _{t \to \infty}
	V(\hat{\theta}(t)) 
	 =& 
	\left\{
	\begin{array}{ll}
	0, & p \in (0,1],
	\\
	\frac{1}{8},   & p =0, 
	\\
	\infty , & p<0.
	\end{array}
	\right.
	\end{align*}

	\end{enumerate}

	\end{prop}

	\begin{proof}
	Denote the  integral of $f(t)$ by 
	\begin{align}
	F(t) :=& \int _{0}^{t} f(\tau ) d\tau
	\nonumber \\
	=& \left\{
	\begin{array}{ll}
	\frac{1}{K(1-p)} \{ (Kt +1) ^{1-p} -1 \},  & p \in \mathbb{R} \setminus \{ 1 \} ,
	\\
	\frac{1}{K} \ln (Kt + 1), & p=1.
	\end{array}
	\right.
	\label{eq:int_f}
	\end{align}
	We only prove the convergence of $V(\hat{\theta}(t)) $. Note that $x=J-n$ and $\hat{x} = J-m $ from the assumption.  
	Then, the explicit solution of $V(\hat{\theta}(t)) $ is 
	\begin{align*}
        V(\hat{\theta} (t) ) = & e^{-2\hat{x}^{2} F(t)} V(\hat{\theta} (0) )  
        \nonumber \\ 
        & 
        + 
        \frac{\hat{x}^{2}}{4} e^{-2\hat{x}^{2} F(t)} \int _{0}^{t} e^{2\hat{x}^{2} F(\tau ) } f(\tau )^2 d\tau 
        .
        \end{align*}
	As $\theta$ is a deterministic uncertain parameter and $\hat{M}(0)$ is given, $V(\hat{\theta} (0) ) =0$.  
	If $p=0$, it implies that $f(t)=1$ and therefore, the claim of the theorem trivially holds.  
	The other cases are as follows.  
	\begin{enumerate}
	\item If $p \in (0, \infty ) $, $f(t)^{2} \leq f(t)$ because $f(t) \in [0,1]$ and $f(t) \leq f(\tau _{0})$ for all $t \geq \tau _{0}$, where $\tau _{0}>0$ is arbitrary chosen.  
	From simple calculation and using the above-mentioned properties, 
	\begin{align*}
	&\int _{0}^{t} e^{2\hat{x}^{2} F(\tau ) } f(\tau )^2 d\tau 
	\\
	=&
	\int _{0}^{\tau _{0}} e^{2\hat{x}^{2} F(\tau ) } f(\tau )^2 d\tau 
	+
	\int _{\tau _{0}}^{t} e^{2\hat{x}^{2} F(\tau ) } f(\tau )^2 d\tau 
	\\
	\leq &
	\int _{0}^{\tau _{0}} e^{2\hat{x}^{2} F(\tau ) } f(\tau ) d\tau 
	+
	f(\tau _{0}) \int _{\tau _{0}}^{t} e^{2\hat{x}^{2} F(\tau ) } f(\tau ) d\tau 
	\\
	=&
	\int _{F(0)}^{F(\tau _{0})} e^{2\hat{x}^{2} s } ds 
	+
	f(\tau _{0}) \int _{F(\tau _{0} ) }^{F(t)} e^{2\hat{x}^{2} s }  ds 
	\\
	=&
	\frac{1}{2\hat{x}^{2}} \Big{\{} 
	e^{2\hat{x}^{2}F(\tau _{0})} - e^{2\hat{x}^{2}F(0)} 
	\\ &
	\quad \quad \quad 
	+
	f(\tau _{0} ) e^{2\hat{x}^{2}F(t)}
	-
	f(\tau _{0} ) e^{2\hat{x}^{2}F(\tau _{0})}
	\Big{\}},
	\end{align*}
	and therefore, 
	\begin{align*}
        V(\hat{\theta} (t) ) \leq & \frac{f(\tau _{0})}{8} (1 - e^{-2\hat{x}^{2} (F(t) - F(\tau _{0}) }  )
        \nonumber \\ 
        &
        + 
        \frac{1}{8}
        e^{-2\hat{x}^{2} F(t)} 
        \Big{\{} e^{2\hat{x}^{2}F(\tau _{0})} - e^{2\hat{x}^{2}F(0)} \Big{\}} . 
        \end{align*}
         As $\tau _{0}$ can be chosen to be arbitrarily large, the first term of the right-hand side of the above equation can be arbitrarily small.  
        As $t \to \infty$, $e^{-2\hat{x}^{2} F(t)} \to 0$ for $p \in (0,1]$ and 
        $e^{-2\hat{x}^{2} F(t)} \to e^{-2\hat{x}^{2}/(K(p-1)) }>0$ for $p>1$.  
        Then, the last term of the right-hand side of the equation remains finite and it is less than $1/8$.  
        Therefore, the claim of the theorem holds for $ p \in (0 , \infty) $.

	\item If $p\in ( -\infty , 0)$, $f(t)^{2} \geq f(t)$ because $f(t) \geq 1$ and $f(t) \geq f(\tau _{0})$ for all $t \geq \tau _{0}$, where $\tau _{0}>0$ is arbitrary chosen.  
	From simple calculation and using the above-mentioned properties, 
	\begin{align*}
	&\int _{0}^{t} e^{2\hat{x}^{2} F(\tau ) } f(\tau )^2 d\tau 
	\\
	\geq &
	\int _{0}^{\tau _{0}} e^{2\hat{x}^{2} F(\tau ) } f(\tau ) d\tau 
	+
	f(\tau _{0}) \int _{\tau _{0}}^{t} e^{2\hat{x}^{2} F(\tau ) } f(\tau ) d\tau 
	\\
	\geq &
	\frac{f(\tau _{0})}{2\hat{x}^{2}} \Big{\{} 
	e^{2\hat{x}^{2}F(t)}
	-
	e^{2\hat{x}^{2}F(\tau _{0})}
	\Big{\}}
	\end{align*}
	and therefore, for $t>\tau _{0}$
	\begin{align*}
	V(\hat{\theta} (t) ) \geq & 
	\frac{f(\tau_{0})}{8}
	\left( 1 - e^{-2\hat{x}^{2} (F(t) - F(\tau _{0}))}  \right)
	\end{align*}
	holds.  
	The second term of the right-hand side of the above inequality vanishes as $t \to \infty$ and 
	because $f(\tau _{0})$ can be arbitrarily large, if $\tau _{0} $ is large, 
	then $V(\hat{\theta} (t) ) \to \infty$ as $t \to \infty$.

\end{enumerate}

	\end{proof}

	From Proposition \ref{prop:statistical_property_tuning}, if $\rho (t) $ and $\hat{\rho}(t)$ are in the same equilibrium state, 
	the parameter $\hat{\theta}(t)$ updated by \eqref{eq:adaptive_tuning} converges to the true value with probability one.  
	Unfortunately, since the true state $\rho (t)$ is not accessible, we cannot confirm 
	whether $\rho(t)$ and $\hat{\rho}(t)$ are practically in the same equilibrium state.  
	To avoid being trapped in different equilibrium points before learning the parameter accurately, 
	we employ feedforward control in the following subsection.  
	
	\begin{rem}
	
	A key to prove Proposition \ref{prop:statistical_property_tuning} is that, $f: [0 , \infty ) \to [0,1]$ is a non-increasing function with $\lim _{t \to \infty} f(t) =0$, while its integral $F(t) = \int _{0}^{t} f(\tau ) d\tau $ becomes a non-decreasing function with $\lim _{t \to \infty} F(t) = \infty$.  
	This is a minor difference from the continuous-time Robbins-Monro algorithms because they require the square integrability of the function $f(t)$ (e.g., \cite[Theorem 1]{chen1994continuous}).  
	Searching for a preferable function for learning $\theta$ is beyond the scope of this study, and
	we only use \eqref{eq:learning_coef} and do not consider other functions.  
	We established some convergence rate problems in \cite{enami2021sice},  which can be referred for details.  

	\end{rem}

\subsection{Local Convergence Property}
	
	In this subsection, we evaluate our tuning algorithm \eqref{eq:adaptive_tuning} with the following control input.  
	\begin{align*}
	u(t) =& u_{FB} (\hat{\rho}(t) ) + u_{FF}(t),
	\end{align*}
	where $u_{FF}(t) \in [0,\infty)$ is 
	a strictly decreasing, bounded continuous function with $\lim_{t \to \infty} u_{FF}(t) =0$,
	and $u_{FB}$ is a stabilizing feedback control law if $\hat{\theta}(t)$ satisfies the condition of Corollary \ref{cor:LiangAminiMason2021} (e.g., of (4.22) or (4.23) in \cite{LiangAminiMason2021}.)  
	The role of $u_{FF}$ is to eliminate the $\hat{\rho}(t)$ from the target state $\rho_{\bar{n}}$ before learning the parameter accurately.  
	Then, our proposed method ensures local convergence under certain assumptions.  
	For convenience, we write $\mathbb{E}_{w}^{\prime}[\bullet ] := \mathbb{E}_{w}[\bullet | \rho(t_{0}), \hat{\rho}(t_{0})]$.  
	
	\begin{thm}
	\label{thm:main_result}

	Let $t_{0}>0$ satisfy 
	$f(t _{0}) < \varepsilon $ for a given sufficiently small $\varepsilon >0$.  
	Considering $\rho (t)$ and $\hat{\rho}(t)$ are the solutions of \eqref{eq:sme} and \eqref{eq:adaptive_sde} 
	starting from $\rho (t_{0})$, $\hat{\rho}(t_{0}) \in \mathcal{S}(\mathbb{\mathbb{C}^{N}})$, respectively, we  
	choose the feedforward control $u_{FF}(t)$ that satisfies 
	$u_{FF}(t) \leq f(t) ^2$ 
	for all $t \geq t_{0}$ 
	and the feedback control $u_{FB}$ that satisfies $\mathbb{E}_{w}^{\prime}[| u_{FB}(\hat{\rho}) |] = O(\varepsilon ^2)$ 
	if $\hat{\rho}$ satisfies $ \mathbb{E}_{w}^{\prime}[\| \hat{\rho} - \rho _{\bar{n}} \| _{\mathrm{Tr}} ]< \varepsilon $.  
	Let $\hat{\theta}(t)$ be the solution of the parameter tuning algorithm \eqref{eq:adaptive_tuning} with $p \in (0.5 ,1]$ and $K>0$ and its initial value $\hat{\theta}(t_{0})$ satisfy $| 1- \hat{\theta}(t_{0}) / \theta | < \varepsilon $.  
	Suppose that $\max \{  \mathbb{E}_{w}^{\prime}[ \| \rho (t) - \rho _{\bar{n}} \| _{\mathrm{Tr}} ], 
	 \mathbb{E}_{w}^{\prime} [ \| \hat{\rho}(t) - \rho _{\bar{n}} \| _{\mathrm{Tr}} ]\} < \varepsilon $ and $\max \{  \mathbb{E}_{w}^{\prime}[ \| \rho (t) - \rho _{\bar{n}} \| _{\mathrm{Tr}}^2 ], 
	 \mathbb{E}_{w}^{\prime} [ \| \hat{\rho}(t) - \rho _{\bar{n}} \| _{\mathrm{Tr}}^2 ]\} < \varepsilon ^2$ for all $t\geq t_{0}$.  
	Besides, assume that the following inequality holds for almost all $t \geq t_{0}$ 
	\begin{align}
	\Delta (\rho (t), \hat{\rho}(t), \hat{\theta}(t) ) \geq 0 \quad \mbox{ a.s.,}
	\label{eq:assumption_main_theorem}
	\end{align}
	and the equality holds iff $V_{\rho(t)}(J_{z}) = V_{\hat{\rho}(t)}(J_{z})=0$, 
	where 
	\begin{align*}
	&\Delta (\rho , \hat{\rho}, \hat{\theta}) 
	\\
	:=& \left(  3V_{\rho}(J_{z}) ^{2} + 2 V_{\rho}(J_{z}) V_{\hat{\rho}}(J_{z}) + 3V_{\hat{\rho}}(J_{z})^2 \right) 
	\nonumber \\
	& - 2 V_{\hat{\rho}} (J_{z}) \left( \mathrm{Tr} [J_{z} (\rho - \hat{\rho})]   +  \mathrm{Tr}[J_{z}\hat{\rho}] \left( 1 - \frac{\hat{\theta}}{\theta}\right) \right)
	,
	\end{align*}
	$V_{\rho}(J_{z}):= \mathrm{Tr}[J_{z}^{2} \rho ] - \mathrm{Tr}[J_{z} \rho]^2$.  
	Then, for $\bar{n}\neq J$,
	\begin{align*}
	\lim _{t \to \infty} (\rho(t) , \hat{\rho}(t)) = (\rho_{\bar{n}}, \rho _{\bar{n}})  \mbox{ and } 
	\lim _{t \to \infty} \hat{\theta}(t) = \theta \ \mbox{a.s.}
	\end{align*}

	\end{thm}

	\begin{proof}
	
	See Appendix.  
	
	\end{proof}
	
	Some readers may think the assumptions of Theorem \ref{thm:main_result} are too strong to be valid in practice; however, several numerical experiments support that they may hold in many cases, one of which is demonstrated in the following section.

\section{NUMERICAL EXPERIMENTS}
\label{sec:numerical_experiments}

	In this section, we examine the proposed method numerically.  
	The dimension of the quantum system is $N=5$, and the Euler-Maruyama method is used with $0.01$ time step width. 
	We use the true parameters as $(\omega , M,\eta ) = (0.5, 1, 0.9)$ and the initial parameters of the adaptive system as $(\hat{\omega} , \hat{M}(0),\hat{\eta} ) = (1,25,1)$, 
	for which the system cannot be stabilized by merely using the feedback control in \cite{LiangAminiMason2021}.  
	The true initial state $\rho (0)$ is randomly generated for each realization and the initial adaptive state is fixed to $\hat{\rho}(0) = \frac{1}{N} I$.  
	The target state $\rho_{\bar{n}}$ is set as $\bar{n}=0$.  
	We set the control inputs as follows.  
	\begin{align*}
	u_{FF}(t) :=  f(t)^2, \quad 
	u_{FB} (\hat{\rho}) :=& 4 (1 - \mathrm{Tr}[\hat{\rho} \rho_{\bar{n}}]) ^{2}
	.  
	\end{align*}
	The parameters of \eqref{eq:learning_coef} are chosen as $(K,p) = (20,0.6)$, and then the simulation is run with 1000 realizations.   
	The results are shown in Figs. \ref{fig_N5K20p06ratio} and \ref{fig_N5K20p06dist}.  
	Fig. \ref{fig_N5K20p06ratio} represents the trajectories of the ratio $\hat{\theta}(t)/\theta$ and 
	Fig. \ref{fig_N5K20p06dist} represents the distance $d(t) = d_{B}((\rho (t) , \hat{\rho}(t)), (\rho_{\bar{n}},\rho_{\bar{n}}))$ \cite{LiangAminiMason2021}, 
	\begin{align*}
	& d_{B}((\rho , \hat{\rho}), (\rho_{n}, \rho_{m}))
	\\
	=&
	\sqrt{2 - 2 \sqrt{ \mathrm{Tr}[\rho \rho_{n}]}} 
	+ 
	\sqrt{2 - 2 \sqrt{ \mathrm{Tr}[\hat{\rho} \rho_{m}] } }
	.  
	\end{align*}
	We also evaluate whether the inequality \eqref{eq:assumption_main_theorem} holds in Fig. \ref{fig_N5K20p06assumption}.   
	From the figures, all sample trajectories of $\hat{\theta}(t)$ and $(\rho(t) , \hat{\rho}(t))$ appear to converge to $\theta$ and the target state $\rho _{0}$, respectively.  
	The inequality \eqref{eq:assumption_main_theorem} sometimes does not hold at the beginning of the simulations, 
	but all sample trajectories satisfy it after $t=450$, shown by the blue dashed line in  Fig. \ref{fig_N5K20p06assumption}, 
	until the states converge to the target states. 
	Moreover, even though our proposed method does not ensure satisfying the condition \eqref{eq:convergence_condition} at all times after a certain point, we confirmed that all sample trajectories of the $\hat{\theta}(t)/\theta$ ratio satisfy the condition of Corollary \ref{cor:LiangAminiMason2021}, i.e., 
	$\hat{\theta}(t)/\theta \in  (1 + \alpha_{0} , 1 + \beta_{0}) \simeq (0.889, 1.11)$, after $t=666$.  
	This result implies that the proposed method ensures that all sample trajectories are in the neighborhood of the true value with a significantly high probability.    
	Although we confirmed that $(K,p)=(20,0.3)$, which does not satisfy the condition of Theorem \ref{thm:main_result}, also works well, but the result is omitted due to the page limitation.

	\begin{figure}[thpb]
       \centering
       \includegraphics[scale=0.4]{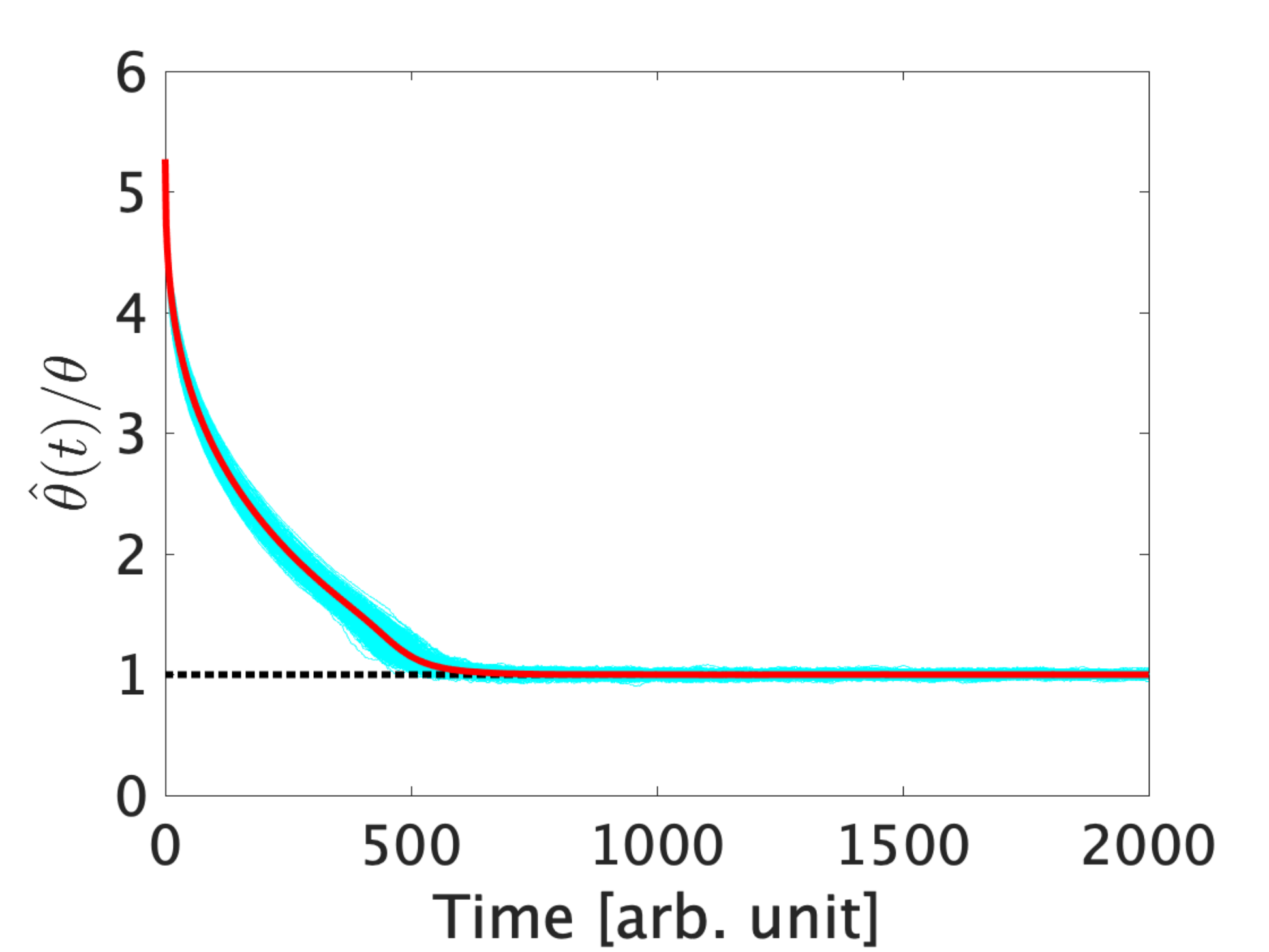}
       \caption{
       The trajectories of the ratio $\hat{\theta} (t) / \theta $ with parameters $(K,p) = (20, 0.6)$.  
       The solid red line represents the average trajectory over 1000 samples
       and the light blue lines represent 1000 sample realizations.  
       }
       \label{fig_N5K20p06ratio}
       \end{figure}
	
	\begin{figure}[thpb]
       \centering
       \includegraphics[scale=0.4]{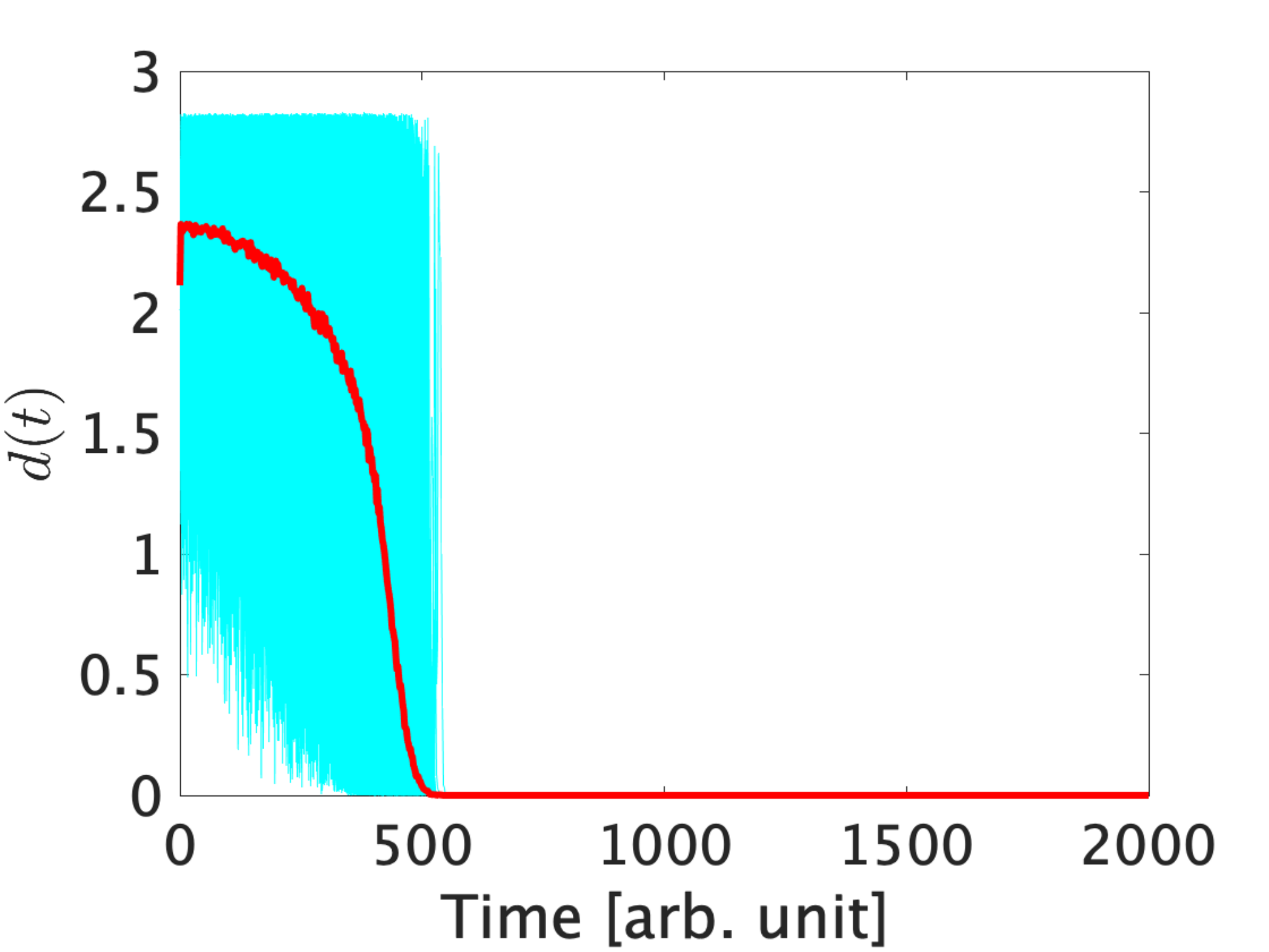}
       \caption{
       The trajectories of $d(t) $ with parameters $(K,p) = $(20, 0.6).  
       The solid red line represents the average trajectory over 1000 samples
       and the light blue lines represent 1000 sample realizations.  }
       \label{fig_N5K20p06dist}
       \end{figure}

	\begin{figure}[thpb]
       \centering
       \includegraphics[scale=0.4]{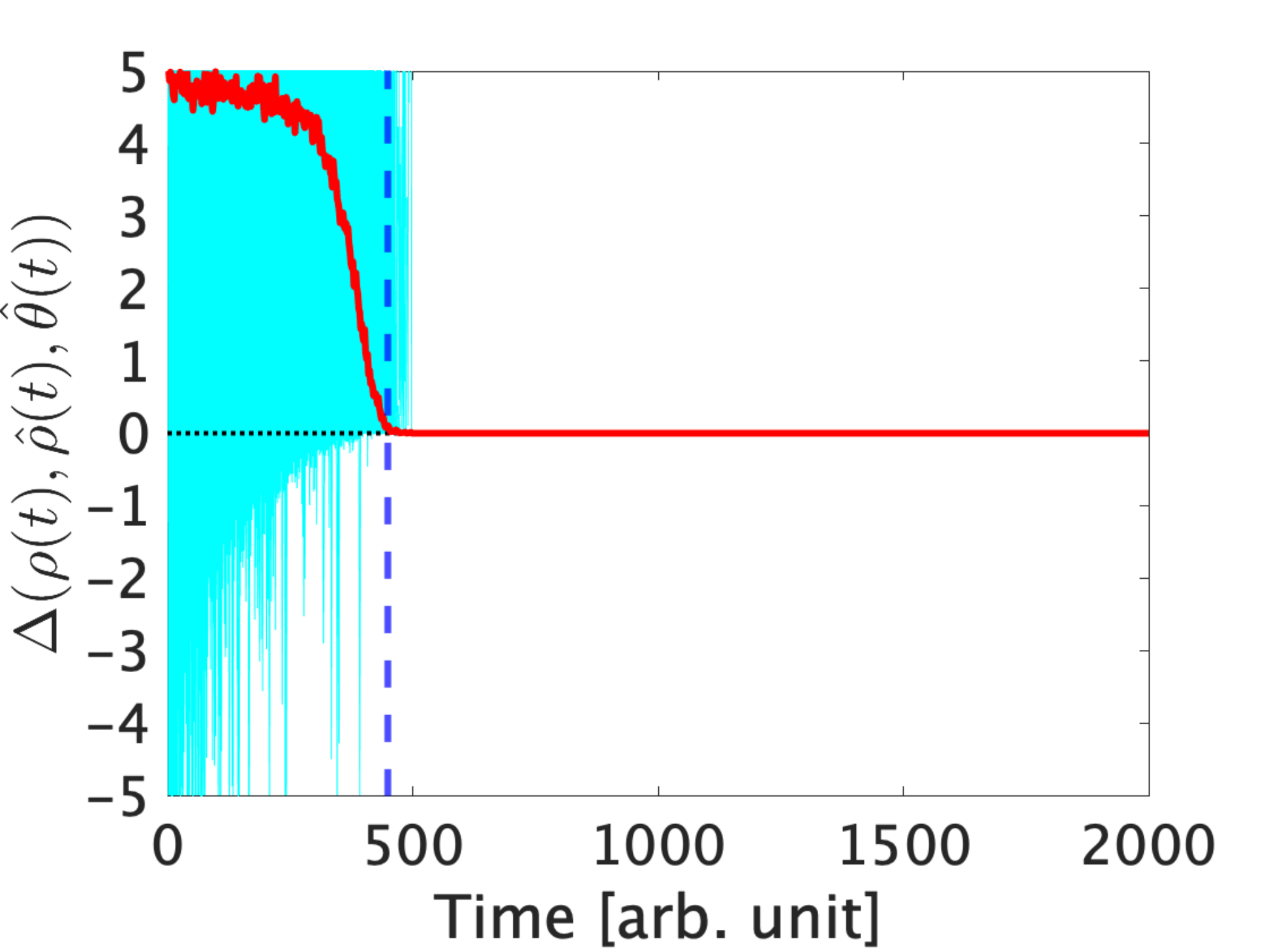}
       \caption{
       The trajectories of $\Delta (\rho (t),\hat{\rho}(t),\hat{\theta}(t)) $ with parameters $(K,p) = $(20, 0.6).  
       The solid red line represents the average trajectory over 1000 samples
       and the light blue lines represent 1000 sample realizations.  The blue dashed line represents the time 450.  }
       \label{fig_N5K20p06assumption}
       \end{figure}

\section{CONCLUSION AND FUTURE WORK}
\label{sec:conclusion}

	In this paper, we proposed an adaptive parameter tuning algorithm for robust stabilizing control of quantum angular momentum systems.  
	The asymptotic property of the estimate and local convergence of the states were evaluated analytically, and numerical experiments show that the proposed method works well for systems with large parametric uncertainty.
	
	The relaxation of Theorem \ref{thm:main_result}'s assumptions and the global convergence property are interesting works for future.

%%%%%%%%%%%%%%%%%%%%%%%%%%%%%%%%%%%%%%%%%%%%%%%%%%%%%%%%%%%%%%%%%%%%%%%%%%%%%%%%
\section*{ACKNOWLEDGMENTS}

The authors gratefully acknowledge the helpful comments and suggestions of the anonymous reviewers.

\appendix

\section*{Proof of Theorem \ref{thm:main_result}}

To prove Theorem \ref{thm:main_result}, we evaluate $G(\rho(t) , \hat{\rho}(t) , \hat{\theta}(t)) := C_{x}(t) + C_{\theta}(t) + V_{\rho (t)} (J_{z}) + V_{\hat{\rho}(t)}(J_{z}) $, where $C_{\theta} (t) := \left| 1 - \frac{\hat{\theta}(t)}{\theta} \right| ^{2}$, and $C_{x} (t) :=| x(t) - \hat{x}(t) |^2$.  
Our proof mainly follows the argument of the proof of Theorem 2.1 in \cite{mao1999stochastic}.

First, we evaluate $V_{\hat{\rho} (t)}(J_{z})$ and $V_{\rho (t)}(J_{z})$.  
\begin{lem}
\label{lem:variance}

 If $\mathbb{E}_{w}^{\prime}[ \| \hat{\rho}(t) - \rho _{\bar{n}}\| _{\mathrm{Tr}} ] < \varepsilon $ holds for some small $\varepsilon >0$, 
 then 
 \begin{align*}
V_{\hat{\rho}(t)}(J_{z}) = \varepsilon 
\hat{\alpha} (t) 
+ 
\varepsilon ^{2} \mathrm{Tr}[ (\hat{\rho}(t) - \rho _{\bar{n}}) J_{z}] ^{2}
,
\end{align*}
where $\hat{\alpha} (t) :=\mathrm{Tr}[ (\hat{\rho}(t) - \rho _{\bar{n}}) (J_{z} - (J -\bar{n})I_{n} )^{2}  ] $ is a nonnegative number and 
$\alpha (t)=0$ iff $\hat{\rho} (t) = \rho _{\bar{n}}$.  

\end{lem}

\begin{proof}
If $\mathbb{E}_{w}^{\prime} [ \| \hat{\rho}(t) - \rho \| _{\mathrm{Tr}} ]  < \varepsilon $ holds for some small $\varepsilon >0$, 
there exists $X(t) = X(t) ^{\ast} \in \mathbb{C} ^{N\times N}$ that satisfies $\hat{\rho} (t) = \rho + \varepsilon X(t)$ and 
$\mathbb{E}_{w}^{\prime} [ \| X(t) \| _{\mathrm{Tr}} ]< 1$.  
This implies that if $\mathbb{E}_{w}^{\prime}[ \| \hat{\rho}(t) - \rho _{\bar{n}} \| _{\mathrm{Tr}} ] < \varepsilon$ holds for a small $\varepsilon >0$, 
then 
\begin{align*}
V_{\hat{\rho}(t)}(J_{z}) = \varepsilon \underbrace{
\mathrm{Tr}[ X(t) (J_{z} - (J -\bar{n})I_{n} )^{2}  ] }_{= \hat{\alpha} (t)}
+ \varepsilon ^{2} \mathrm{Tr}[ X(t) J_{z}] ^{2}
\end{align*}
holds.  Since $ \rho _{\bar{n}} + \varepsilon X(t) \in \mathcal{S}(\mathbb{C}^{N}) $, the $(\bar{n}+1)$-th diagonal element of $X(t)$ needs to be nonpositive and the other diagonal elements are nonnegative, and $\mathrm{Tr}[X(t)] =0$.  
The $(\bar{n}+1)$-th diagonal element of $(J_{z} - (J -\bar{n})I_{n} )^{2}$ becomes $0$, 
so $\hat{\alpha} (t)$ is nonnegative and $\hat{\alpha} (t)=0$ iff $\hat{\rho} (t) = \rho _{\bar{n}}$.  

\end{proof}

Therefore, $\mathbb{E}_{w}^{\prime}[V_{\hat{\rho}(t)} (J_{z}) ] = O(\varepsilon )$.  
Similar argument gives $\mathbb{E}_{w}^{\prime}[V_{\hat{\rho}(t)} (J_{z})^2 ] = O(\varepsilon ^2)$ and $\mathbb{E}_{w}^{\prime}[V_{\rho(t)}(J_{z}) ]=O(\varepsilon)$ from the assumptions.  
Note that $dy(t)$ in \eqref{eq:adaptive_sde} can be replaced by $\theta x(t) dt + dw(t)$.  
Let $\mathcal{L}$ be the infinitesimal generator 
 \cite{kushner1967ssa}.  
Using the classical Ito calculus, the infinitesimal generators of $V_{\rho(t)}(J_{z})$ and $V_{\hat{\rho}(t)}(J_{z})$ are as follows.  
\begin{align*}
\mathcal{L} V_{\rho(t)} (J_{z}) 
=& 
-4\theta^2 V_{\rho(t)} (J_{z}) ^2 
-\mathrm{i} u(t) \mathrm{Tr}[ J_{y} [J_{z} , \rho (t) ]_{-} ] ,
\\
\mathcal{L} V_{\hat{\rho}(t)} (J_{z}) 
\leq &
-4\theta ^2 V_{\hat{\rho}(t)} (J_{z}) ^2 
- \mathrm{i} u(t) \mathrm{Tr}[ J_{y} [J_{z} , \hat{\rho} (t) ]_{-} ] 
\\ & 
+ 2\theta ^2 V_{\hat{\rho}(t)} (J_{z}) (x(t)- \hat{x}(t) ) 
\\ &
+2 \theta \hat{x}(t) V_{\hat{\rho}(t)} (J_{z}) (\theta - \hat{\theta}(t)  ) 
\\
&
+4( \theta ^2 -\hat{\theta}(t)^2) V_{\hat{\rho}(t)} (J_{z}) ^2 
\\ & 
+2(\hat{\theta}(t) -\theta ) V_{\hat{\rho}(t)} (J_{z}) 
(\theta x(t)- \hat{\theta}(t) \hat{x}(t) ) 
.
\end{align*}
Since $\mathbb{E}_{w}^{\prime} [V_{\rho(t)} (J_{z})^2 ] = O(\varepsilon ^2)$, $\mathbb{E}_{w}^{\prime} [ [ \rho(t) , J_{z}]_{-} ]=O(\varepsilon)$, $\mathbb{E}_{w}^{\prime}[ [ \hat{\rho}(t) , J_{z}]_{-} ] =O(\varepsilon)$, 
$\mathbb{E}_{w}^{\prime} [u_{FB}(\hat{\rho}(t)) ]= O(\varepsilon ^2 )$, and $u_{FF} (t)= O(f(t)^{2})$, 
\begin{align}
& \mathbb{E}_{w}^{\prime}[ \mathcal{L} (V_{\rho(t)} (J_{z}) + V_{\hat{\rho}(t)} (J_{z}) ) ]
\nonumber
\\
= & \mathbb{E}_{w}^{\prime} \Bigg{[}
- 4\theta ^2 (V_{\rho(t)} (J_{z}) ^2 + V_{\hat{\rho}(t)} (J_{z}) ^2 ) + \sigma_{1}u_{FF}(t)
\nonumber \\
&
+ 2\theta ^2 V_{\hat{\rho}(t)} (J_{z}) \left( (x(t) - \hat{x}(t))  + \hat{x}(t) \left( 1 - \frac{\hat{\theta}(t)}{\theta} \right)  \right)
\Bigg{]}
\nonumber
\\
&
+O(\varepsilon ^3 ) 
+O(\varepsilon C_{\theta} (t) ) + O\left( \varepsilon ^2 \sqrt{C_{\theta}(t)} \right)
,
\label{eq:orderestimate1}
\end{align}
where $\sigma _{1} := \max _{\rho \in \mathcal{S}(\mathbb{C}^{N}) } |\mathrm{Tr}[J_{y} [J_{z},\rho ]_{-}]|$.

Next, we calculate the infinitesimal generator of $C_{x}(t)$ and $C_{\theta}(t)$.   
From simple calculation, 
\begin{align*}
&\mathcal{L}C_{x}(t) 
 \\
\leq & 
2 |u_{FB}(\hat{\rho}(t))| 
 \left| \mathrm{Tr}[  J_{y} [ \rho(t) - \hat{\rho}(t) , J_{z} ]_{-} ] \right|   \sqrt{C_{x}(t)}
\\ & 
+ 8  J \sigma _{1} u_{FF}(t)
+
4 \hat{\theta}(t) \theta V_{\hat{\rho}(t)}(J_{z})
 \\
 & \quad \times
 \Bigg{\{}
 -C_{x}(t)
 +
 |\hat{x}(t)| \sqrt{C_{\theta} (t)} \sqrt{C_{x}(t)} 
 \Bigg{\}}
 \\
& + 
\left( 
\theta  V_{\rho(t)}(J_{z})
 - 
 \hat{\theta} (t)  V_{\hat{\rho}(t)}(J_{z})
 \right) ^{2} 
 .
\end{align*}
From the definition of $C_{\theta}(t)$, 
\begin{align*}
 \mathcal{L} C_{\theta}(t) 
\leq 
& 2f(t) 
\Bigg{\{} -  \hat{x}(t) ^{2} C_{\theta}(t) +  |\hat{x}(t) |  \sqrt{C_{\theta} (t)} \sqrt{C_{x}(t)} 
\\
& \hspace{2cm} 
+ \frac{\hat{x}(t)^{2} f(t)}{8\theta ^2} \Bigg{\}} 
.
\end{align*}
Since the expectation of the right-hand side of the above inequality is at most $O(\varepsilon ^2)$ for small $t - t_{0}> 0$, 
$\mathbb{E}_{w}^{\prime} [ C_{\theta}(t) ] - C_{\theta}(t_{0}) = 
\int _{t_{0}}^{t} \mathbb{E}_{w}^{\prime}[ \mathcal{L} C_{\theta}(\tau ) ] d\tau 
\leq (t-t_{0}) \times O(\varepsilon ^{2}) $, where the Dynkin's formula \cite{kushner1967ssa} is used.  
Let $a(t) := 4 \hat{\theta}(t) \theta V_{\hat{\rho}(t)}(J_{z})$ and $b(t) := 2f(t)$.  
Note that $a(t)=0$ iff $V_{\hat{\rho}(t)}(J_{z})=0$.  
Then, 
\begin{align}
& \mathbb{E}_{w}^{\prime}[ \mathcal{L}(C_{x}(t) + C_{\theta}(t)) ]
\nonumber
\\
\leq &
\mathbb{E}_{w} ^{\prime}
\Bigg{[}
2 |u_{FB}(\hat{\rho}(t))| 
 \left| \mathrm{Tr}[  J_{y} [ \rho(t) - \hat{\rho}(t) , J_{z} ]_{-} ] \right|   \sqrt{C_{x}(t)}
 \nonumber \\ & \quad \quad 
+ 
8  J \sigma _{1} u_{FF}(t)
\nonumber
\\
& 
\quad \quad 
-a(t) C_{x}(t) - b(t) \hat{x}(t)^{2} C_{\theta}(t) 
\nonumber
\\
& 
\quad \quad 
+ (a(t) + b(t) )|\hat{x}(t)| \sqrt{C_{x}(t) C_{\theta}(t)}
\nonumber
 \\
& 
\quad \quad 
+ 
\left( 
\theta  V_{\rho(t)}(J_{z})
 - 
 \hat{\theta} (t)  V_{\hat{\rho}(t)}(J_{z})
 \right) ^{2}
+ \frac{J^{2} b(t)^2}{16\theta ^2} 
\Bigg{]}
\nonumber
\\
= & 
\mathbb{E}_{w} ^{\prime} \Bigg{[}
\underbrace{
\theta ^{2}\left( 
 V_{\rho(t)}(J_{z}) -  V_{\hat{\rho}(t)}(J_{z})
 \right) ^{2}
  }_{O(\varepsilon ^2)}
  \nonumber \\
  & \quad \quad 
 +
 \underbrace{
 \frac{J^{2} b(t)^2}{16\theta ^2}
+ 8  J \sigma _{1} u_{FF}(t)
}_{O(f(t)^2)}
\nonumber
\\
& \quad \quad 
-
 \big{(} 
 \underbrace{ a(t) C_{x}(t)   }_{=O(\varepsilon ^3)}
+ 
\underbrace{b(t) | J-\bar{n} |^{2}C_{\theta}(t) }_{=O(f(t) C_{\theta}(t))}\big{)} 
+O(\varepsilon ^4) 
\Bigg{]}
.  
\label{eq:orderestimate2}
\end{align}
Note that 
\[ \mathbb{E}_{w}^{\prime} [ |u_{FB}(\hat{\rho}(t))| 
 \left| \mathrm{Tr}[  J_{y} [ \rho(t) - \hat{\rho}(t) , J_{z} ]_{-} ] \right|   \sqrt{C_{x}(t)} ] = O(\varepsilon ^4).\]

From \eqref{eq:orderestimate1} and \eqref{eq:orderestimate2},  
\begin{align*}
&\mathbb{E} _{w}^{\prime}\left[ \mathcal{L}G(\rho(t) , \hat{\rho}(t) , \hat{\theta}(t)) \right]
\\
\leq &
\mathbb{E} _{w}^{\prime} \Bigg{[}
\theta ^{2}\left( 
 V_{\rho(t)}(J_{z}) -  V_{\hat{\rho}(t)}(J_{z})
 \right) ^{2}
 +
 \frac{J^{2} b(t)^2}{16\theta ^2} 
 \\
 & \quad \quad 
   + (1+8J) \sigma_{1}u_{FF}(t)
   \\
   & \quad \quad 
 - 4 \theta ^{2} \left( V_{\rho (t)} (J_{z}) ^2 + V_{\hat{\rho} (t)} (J_{z}) ^2 \right) 
\\
& \quad \quad 
+ 2\theta ^2 V_{\hat{\rho}(t)} (J_{z}) 
\\ & \quad \quad \quad 
\times \left( (x(t) - \hat{x}(t))  + \hat{x}(t) \left( 1 - \frac{\hat{\theta}(t)}{\theta} \right)  \right)
\\
& \quad \quad 
-b(t) | J-\bar{n} |^{2}C_{\theta}(t) +O(\varepsilon ^3) \Bigg{]}
\\
=&
- \mathbb{E} _{w} ^{\prime} \left[ \Delta (\rho (t), \hat{\rho}(t), \hat{\theta}(t) ) 
+
b(t) | J-\bar{n} |^{2}C_{\theta}(t) \right]
\\
&
+
\gamma (t) +O(\varepsilon ^3)
,  
\end{align*}
where $\gamma (t):= \frac{J^{2} b(t)^2}{16\theta ^2} + (1+8J) \sigma_{1}u_{FF}(t) $.  
As $p \in (0.5 ,1]$, $b(t)^{2} = 4f(t)^2$ and $u_{FF}(t)$ are integrable, i.e., $\gamma (t)$ is integrable.   

Together with the assumption \eqref{eq:assumption_main_theorem}, $\mathbb{E}_{w}^{\prime}[C_{\theta}(t)] \leq O(\varepsilon ^2)$ for all $t\geq t_{0}$ and using Dynkin's formula \cite{kushner1967ssa}, 
\begin{align*}
& 
\mathbb{E}_{w}^{\prime} \left[
G(\rho(\infty) , \hat{\rho}(\infty) , \hat{\theta}(\infty)) 
\right]
- 
G(\rho(t_{0}) , \hat{\rho}(t_{0}) , \hat{\theta}(t_{0}))
\\
&+ 
\int _{t_{0}}^{\infty}
\mathbb{E}_{w} ^{\prime}[ \Delta (\rho (\tau), \hat{\rho}(\tau), \hat{\theta}(\tau) ) ] d\tau 
\\
&
+ |J-\bar{n}|^2 \int _{t_{0}}^{\infty} b(\tau) \mathbb{E}_{w}^{\prime}[C_{\theta}(\tau) ] d\tau 
+ \int _{t_{0}}^{\infty} \mathbb{E}_{w}^{\prime}[ O(\varepsilon ^3) ] d\tau 
\\
\leq &\int _{t_{0}}^{\infty} \gamma (\tau ) d\tau < \infty
\end{align*}
holds.  
Note that the integrand of the last term of the left-hand side of the first inequality $\mathbb{E}_{w}^{\prime}[O(\varepsilon ^3)]$ converges to zero faster than the other terms.  
The other terms of the left-hand side are positive and need to be finite.  
Hence, $\lim _{t \to \infty}\Delta (\rho (t), \hat{\rho}(t), \hat{\theta}(t) ) =0$ a.s.  
Since $x(t)$ or $\hat{x}(t)$ fluctuates randomly if $\rho (t)\neq 0$ or $\hat{\rho}(t) \neq 0$, 
$\lim _{t\to \infty}\Delta (\rho (t) , \hat{\rho}(t),\hat{\theta}(t))=0$ implies that $V_{\rho (t)} (J_{z}) $ and $V_{\hat{\rho}(t)}(J_{z})$ converge to the origin.  
From the assumption that $\rho (t) $ and $\hat{\rho}(t)$ stay in the neighborhood of $\rho _{\bar{n}}$, 
the convergence of $V_{\rho (t)} (J_{z}) $ and $V_{\hat{\rho}(t)}(J_{z})$ implies $( \rho (t), \hat{\rho}(t))$ converges to $(\rho _{\bar{n}} , \rho _{\bar{n}})$.   
Furthermore, since $b(t)$ is not integrable and, although we skip the proof, $\hat{\theta}(t)$ is continuous in $t$, $\lim_{t \to\infty} C_{\theta}(t) =0 $ a.s.  
Therefore, $\lim _{t \to \infty} \hat{\theta}(t) = \theta$ a.s.

\end{document}